\documentclass[12pt,twoside]{amsart}
\usepackage[latin1]{inputenc}
\usepackage{times}
\usepackage{amsthm}
\usepackage{amssymb, verbatim}
\usepackage{amsmath}
\usepackage{mathrsfs}
\usepackage{hyperref}

\topmargin = - 40 pt %distanza dal margine superiore 0= un bel po` negativo diminuisce
\usepackage{calc}
\usepackage{graphicx}

\usepackage[english]{babel}
\usepackage{latexsym}

\usepackage{amsfonts}
\usepackage{syntonly}
\usepackage{graphicx}
\usepackage{mathptmx}
\usepackage{subfig}
\usepackage{url}
\usepackage{lscape}
\usepackage[letterpaper]{geometry}
\input xy
\xyoption{all}

\newtheorem{theorem}{Theorem}[section]

\newtheorem{lemma}[theorem]{Lemma}

\newtheorem{proposition}[theorem]{Proposition}

\newtheorem{cor}[theorem]{Corollary}

\newcommand     {\Rset}    {{\mathbb R}}

\newcommand     {\Nset}    {{\mathbb N}}

\DeclareMathOperator{\less} {\textit{less}}
\DeclareMathOperator{\more} {\textit{more}}

\def\quotient#1#2{%
    \raise1ex\hbox{$#1$}\Big/\lower1ex\hbox{$#2$}%
}

\begin{document}

\title[The cardinality of Kiselman's semigroups grows double-exponentially]
{The cardinality of Kiselman's semigroups\\ grows double-exponentially}

\author[A.~D'Andrea]{Alessandro D'Andrea}
%\thanks{The author was partially supported by
%PRIN ``Spazi di Moduli e Teoria di Lie'' fundings from MIUR and
%project MRTN-CT 2003-505078 ``LieGrits'' of the European Union}
%\address{Dipartimento di Matematica, Universit\`a degli Studi di
%Roma ``La Sapienza'', Roma}
\email{dandrea@mat.uniroma1.it}

\author[S.~Stella]{Salvatore Stella}
\email{salvatore.stella@univaq.it}

\date{\today}

\begin{abstract}
Let $K_n$ denote Ganyushkin-Kudryavtseva-Mazorchuk's generalization of Kiselman's semigroups. We show that the sequence $2^{-n/2}\cdot \log|K_n|$ admits finite limits as $n$ grows to infinity both on odd and even values.
\end{abstract}



\maketitle

%\tableofcontents

\section{Introduction}

Kiselman's semigroup $K_n$ is a generalization \cite{km} of a monoid introduced by Kiselman~\cite{k} in a convexity theory setting. It is related to $0$-Hecke algebras \cite{n}, arises as a quotient of the Richardson-Springer monoid \cite{rs}, describes the evolution properties of  certain graph dynamical systems \cite{cd}, and exhibits interesting combinatorics.

In \cite{km}, Kudryavtseva and Mazorchuk show that $K_n$ for $n \in \Nset$, is always finite by providing the explicit, quite broad, upper bound
$$|K_n| \leq 1 + n^{L(n)}$$
for their cardinality, where
$$L(n) = \begin{cases}
2^{k+1}-2, \qquad \mbox{ if } n = 2k\\
3 \cdot 2^k - 2, \qquad \mbox{ if } n = 2k+1.
\end{cases}$$
However, data from \cite{OEIS} suggest that the growth of $|K_n|$ is sharply double exponential in $n$.

In this short note, we would like to provide easy double exponential upper and lower bounds for $|K_n|$ and show that the sequence $ 2^{-n/2} \log|K_n|$ has a finite limit when restricted to either even or odd values of $n$. Throughout the paper, $\log$ denotes the base $2$ logarithm.

\section{Canonical words}

Henceforth $A$ will be a totally ordered set. The Kiselman semigroup $K(A)$ 
is defined by the monoid presentation
\begin{equation}\label{presentation}
K(A) = \langle a \in A \,\,|\,\, a^2 = a,\,\,  aba = bab = ab, \mbox{ for all } a<b\rangle.
\end{equation}
We will be concerned with the case where $A$ is finite. Then $K(A)$ depends, up to isomorphism, only on the cardinality of $A$. When $A = \{1, \dots, n\}$ with the standard total order, then $K(A)$ is usually denoted $K_n$; in this case, generators of $K(A)$ are denoted by $a_i$ to ease readability.

Let $W(A)$ be the monoid of finite words on the alphabet $A$. Then for every $x \in K(A)$ there exists a unique \cite{ad, km} shortest element (i.e., reduced expression) in $\pi^{-1}(x)$, where $$\pi: W(A) \to K(A)$$ is the canonical projection. Such reduced expressions admit a combinatorial description: $w \in W(A)$ is the shortest element in some $\pi^{-1}(x)$ if and only if whenever $a u a$ is a subword of $w$, where $u \in W(A)$ and $a \in A$, then $u$ contains both some letter $b>a$ and some letter $c<a$; elements of $W(A)$ satisfying the above conditions are called {\em canonical words} and elements in $K(A)$ are thus in bijection with canonical words in $W(A)$.

Notice that, due to the presentation \eqref{presentation}, the set of letters showing up in each word lying in $\pi^{-1}(x)$ only depends on $x\in K(A)$. Also denote by $\less(a)$ and $\more(a)$ the number of elements in $A$ that are less or more than $a$. The following easy facts are proved in \cite{km}.
\begin{proposition}\label{standard}
Let $w\in W(A)$ be the unique reduced expression of an element $x \in K(A)$. Then:
\begin{itemize}
\item each letter $a \in A$ occurs in $w$ at most $2^{\less(a)}$ times;
\item each letter $a \in A$ occurs in $w$ at most $2^{\more(a)}$ times;
%\item when $i \leq \lceil{\frac{n}{2}}\rceil$, the letter $a_i$ occurs in $w$ at most $2^{i-1}$ times;
%\item when $i \geq \lceil{\frac{n+1}{2}\rceil}$, the letter $a_i$ occurs in $w$ at most $2^{n-i}$ times;
\item the length of $w$ is less or equal than $L(|A|)$.
\end{itemize}
\end{proposition}

\section{A bound from below}

Providing a double exponential lower bound for the cardinality $|K_n|$ of Kiselman's semigroups is easy.
%Denote by $K(n)$ the cardinality of $K_n$. 
\begin{proposition}\label{below}
One has $|K_{n+2}| \geq 2 |K_{n}|^2$ for every $n \in \Nset$. In particular,
\begin{equation}\label{increasing}
2^{-(n+2)/2}\log( 2|K_{n+2}|) \geq 2^{-n/2}\log (2|K_n|),
\end{equation}
for every $n \in \Nset$.
\end{proposition}
\begin{proof}
For every choice of canonical words $w, w' \in W(\{a_2, \dots, a_{n+1}\})$, words of the form $wa_1 a_{n+2}w'$ and $w a_{n+2} a_1 w'$ are still canonical. This shows that the number of canonical words in $n+2$ letters is at least twice the square of the number of canonical words in $n$ letters, thus yielding $|K_{n+2}| \geq 2 |K_{n}|^2$.

We may write this as $2|K_{n+2}| \geq (2|K_n|)^2\!$, or equivalently $\log(2|K_{n+2}|) \geq 2 \log(2|K_n|)$, which immediately proves \eqref{increasing}.
\end{proof}

\begin{cor}\label{atleast}
The cardinality of $K_n$ grows at least double exponentially.
\end{cor}
\begin{proof}
As $|K_0| = 1, |K_1| = 2$, we immediately see that %$\log(2K(0))2^{-0/2} = 1$, $\log(2K(1))2^{-1/2} = 2^{1/2}$, thus
$$|K_n| \geq \frac{1}{2} \cdot 2^{2^{\lceil{n/2}\rceil}}.$$
\end{proof}

\section{A bound from above}

We will use the following well known estimates.
\begin{lemma}
One has:
\begin{enumerate}\label{binomials}
\item\label{dueuno}
${2N \choose N} \leq \frac{2^{2N}}{\sqrt{\pi N}};$
\item\label{quattrotre}
${4N \choose N}{3N \choose N} \leq \frac{2^{6N}}{\pi N \sqrt{2}};$
\item\label{treuno}
${3N \choose N} \leq \left(\frac{27}{4}\right)^N,$
\end{enumerate}
for every positive $N \in \Nset$.
\end{lemma}
\begin{proof}
Part (1) is standard, and a well known proof\footnote{Some elementary proofs for stronger, and asymptotically optimal, estimates for the middle binomial coefficients may be found in \cite{MSE}.} follows from observing that
$${2N \choose N} = \frac{2^{2N}}{\pi}\int_{-\pi/2}^{\pi/2} (\cos x)^{2N} dx \leq 
\frac{2^{2N}}{\pi}\int_{-\pi/2}^{\pi/2} e^{-Nx^2} dx \leq \frac{2^{2N}}{\pi}\int_{-\infty}^{+\infty} e^{-Nx^2} dx = \frac{2^{2N}}{\sqrt{\pi N}}.$$
Then claim (2) is a consequence of part (1) as in 
$${4N \choose N}{3N \choose N} = {4N \choose 2N}{2N \choose N}\leq \frac{2^{4N}}{\sqrt{2\pi N}} \frac{2^{2N}}{\sqrt{\pi N}} = \frac{2^{6N}}{\pi N \sqrt{2}},$$
whereas (3) is immediately proved by induction.
\end{proof}

\begin{lemma}\label{howmanywords}
Let $n = 2k$ be an even number. The number of words in $W(\{a_1, \dots, a_n\})$ of length~$L(n)$ in which each $a_i, 1 \leq i \leq n,$ occurs $\min(2^{i-1}, 2^{n-i})$ times is bounded from above~by
$$\frac{2^{6 \cdot (2^{k} - 1)}}{(\pi\sqrt{2})^{k} 2^{k(k-1)/2}}.$$
\end{lemma}
\begin{proof}
We need to prove the above upper bound for the multinomial coefficient
\begin{equation}\label{multinomial}
{L(n) \choose 2^{k-1}, 2^{k-1}, 2^{k-2}, 2^{k-2}, \dots, 2, 2, 1, 1} = \prod_{h=1}^k \left({2 \cdot 2^{h} - 2 \choose 2^{h - 1}}\cdot{3 \cdot 2^{h - 1} - 2\choose 2^{h - 1}}\right).
\end{equation}
Now, notice that 
\begin{equation*}
{2 \cdot 2^{h} - 2 \choose 2^{h - 1}}{3 \cdot 2^{h - 1} - 2\choose 2^{h - 1}} \leq
{2 \cdot 2^{h} \choose 2^{h - 1}}{3 \cdot 2^{h - 1}\choose 2^{h - 1}}
= {4 \cdot 2^{h - 1} \choose 2^{h - 1}}{3 \cdot 2^{h - 1}\choose 2^{h - 1}}.
\end{equation*}
We may then use Lemma \ref{binomials}\eqref{quattrotre} to obtain
$${4 \cdot 2^{h - 1} \choose 2^{h - 1}}{3 \cdot 2^{h - 1}\choose 2^{h - 1}} \leq \frac{2^{3 \cdot 2^h}}{\pi 2^{h - 1} \sqrt{2}}.$$
Substituting this into \eqref{multinomial} yields the claim.
\end{proof}

\begin{proposition}\label{above}
Let $n = 2k$ be an even number. Then $|K_n| \leq 2^{6 \cdot 2^k}$.
\end{proposition}
\begin{proof}
The cardinality of $K_{n}$ coincides with the number of canonical words on $n$ letters. Thanks to Proposition \ref{standard}, we know that each such canonical word arises, non uniquely, as a prefix of a word of length $L(n)$ in which $a_i$ occurs $\min(2^{i-1}, 2^{n-i})$ times. We may thus bound $|K_n|$ with the number of all such prefixes, possibly with repetition.

As a word of length $\ell$ has $\ell+1$ prefixes, when $n = 2k$ is even, we obtain
%$$K(n) \leq (2 \cdot 2^k - 1) \cdot {2 \cdot 2^{k} - 2 \choose 2^{k - 1}}{3 \cdot 2^{k - 1} - 2\choose 2^{k - 1}}\dots{14 \choose 4}{10 \choose 4}{6 \choose 2}{4 \choose 2}{2 \choose 1}{1 \choose 1}.$$
\begin{eqnarray*}
|K_n| & \leq & (L(n) + 1) \cdot {L(n) \choose 2^{k-1}, 2^{k-1}, 2^{k-2}, 2^{k-2}, \dots, 2, 2, 1, 1}\\
& \leq & 2^{k+1} \cdot  \frac{2^{6 \cdot 2^{k} - 6}}{(\pi\sqrt{2})^{k} 2^{k(k-1)/2}}
= 2^{6 \cdot 2^k - (k^2-2k + 10)/2 - k \log \pi}
\leq 2^{6 \cdot 2^k}.
\end{eqnarray*}
\end{proof}

\begin{proposition}\label{odd}
Let $n = 2k+1$ be an odd number. Then $|K_n|\leq 2^{c \cdot 2^{n/2}}$, where $c = \log(432)/\sqrt{2}$.
\end{proposition}
\begin{proof}
Starting with the same considerations as in the previous proof, one can use Lemma \ref{howmanywords}, Lemma \ref{binomials}\eqref{treuno}, and the identity $L(2k+1)=L(2k)+2^k=3 \cdot 2^k - 2$ to compute
\begin{eqnarray*}
  |K_n| & \leq & (L(n)+1) \cdot {L(n) \choose 2^k, 2^{k-1}, 2^{k-1}, 2^{k-2}, 2^{k-2}, \dots, 2, 2, 1, 1}\\
  & = & (L(n)+1) \cdot {3 \cdot 2^k - 2 \choose 2^k} \cdot {L(2k) \choose 2^{k-1}, 2^{k-1}, 2^{k-2}, 2^{k-2}, \dots, 2, 2, 1, 1}\\
  & \leq & 3\cdot 2^k \cdot {3 \cdot 2^k \choose 2^k} \cdot \frac{2^{6 \cdot (2^{k} - 1)}}{(\pi\sqrt{2})^{k} 2^{k(k-1)/2}}\\
  & \leq & 3\cdot 2^k \cdot \left(\frac{27}{4}\right)^{2^k} \cdot \frac{2^{6 \cdot (2^{k} - 1)}}{(\pi\sqrt{2})^{k} 2^{k(k-1)/2}}\leq 2^{c \cdot 2^{n/2}},%\\
%& \leq & \frac{3}{2}\left(\frac{27}{4}\right)^{2^k}\cdot 2^{6 \cdot 2^k + k - k^2/2 - 5 - k \log \pi} \leq 2^{\frac{\log_2 432}{\sqrt{2}} \cdot 2^{n/2}},
\end{eqnarray*}
where $c = (\log 432)/\sqrt{2} \sim 6.1906\,\,$.
\end{proof}

\begin{theorem}
The following asymptotics hold:
$$\log |K_{2k}| \sim c_{even} 2^k, \qquad \log |K_{2k+1}| \sim c_{odd} 2^{(2k + 1)/2},$$
where $c_{even} \leq 6$ and $c_{odd} \leq (\log_2 432)/\sqrt{2}$ are positive constants.
\end{theorem}
\begin{proof}
The sequence $2^{-k}\log(2|K_{2k}|)$ --- respectively $2^{-(2k+1)/2}\log(2|K_{2k+1}|)$ --- is increasing due to Proposition \ref{below}, so it converges to a finite limit as soon as it is bounded. However, the asymptotically equivalent sequence $2^{-k}\log|K_{2k}|$ (resp. $2^{-(2k+1)/2}\log|K_{2k+1}|$) is bounded by Proposition \ref{above} (resp. Proposition \ref{odd}), and the claim then follows easily.
\end{proof}

We expect, but cannot prove, that $c_{even} = c_{odd}$, i.e., that $\log |K_n| \sim c 2^{n/2}$ for some positive $c \in \Rset$. Notice that the embeddings $K_{2k-1} \subset K_{2k}\subset K_{2k+1}$ imply that
$$\frac{\log |K_{2k-1}|}{\log |K_{2k}|} \leq 1 \leq \frac{\log |K_{2k+1}|}{\log |K_{2k}|}$$
and in the limit
$$c_{odd}/\sqrt{2} \leq c_{even} \leq c_{odd} \sqrt{2}.$$

\section{Structure of longest words in $K_{2k+1}$}

In this section, we prove a technical statement showing that the lower bound in Corollary \ref{atleast} is established, in the odd $n = 2k+1$ case, by the precise number of longest words in $K_{2k+1}$. First of all, we recall that Proposition \ref{standard} is easily proved by induction from the following observation.

\begin{lemma}\label{plusone}
Let $w \in W(A)$ be a canonical word. Then the number of occurrences of $a \in A$ in $w$ is less or equal to the total number of occurrences of letters less than $a$ in~$w$, increased by $1$. A similar bound holds for letters more than $a$ in $w$.
\end{lemma}
\begin{proof}
As $w$ is canonical, between any two occurrences of $a$ in $w$ there must lie both a lower and a higher letter.
\end{proof}

\begin{proposition}
Let $n = 2k+1>1$ be an odd number. Then all canonical words of length $L(n)$ in $W(\{a_1, \dots, a_n\})$ are of the form
$$w a_1 a_{n} w' \qquad \mbox{ or } \qquad w a_{n} a_1 w',$$
where $w, w'$ are canonical words of length $L(n-2)$ in $W(\{a_2, \dots, a_{n-1}\})$.
\end{proposition}
\begin{proof}
First of all, if $w, w'$ are canonical words of length $L(n-2)$ in $W(\{a_2, \dots, a_{n-1}\})$ then $w a_1 a_{n} w'$ is still canonical, of length $2L(n-2) + 2 = L(n)$. As $L(1) = 1$, and since words of length $1$ are trivially canonical, an easy induction shows that $W(\{a_1, \dots, a_n\})$ indeed contains words of length $L(n)$. Moreover, in any such word, each letter $a_i$ occurs precisely $2^{\min(i-1, n-i)}$ times, as these numbers add up to $L(n)$.

Let now $u$ be a canonical word of length $L(n)$ in $W(\{a_1, \dots, a_n\})$. Both $a_1$ and $a_n$ occur exactly once, and we may assume without loss of generality that $a_1$ occurs on the left of $a_n$, so that $u = w a_1 u'$, where $a_n$ occurs in the subword $u'$. Notice that both $w$ and $u'$ need to be canonical words.

The number of occurrences of $a_i$ for $2\leq i \leq k+1$ in $u$ is exactly $2^{i-1}$, whereas it occurs in $w \in W(\{a_2, \dots, a_{n-1}\})$ and in $u' \in W(\{a_2, \dots, a_n\})$ at most $2^{i-2}$ times. This shows that, for $2\leq i \leq k+1$, $a_i$ occurs in both $w$ and $u'$ precisely $2^{i-2}$ times. By Lemma \ref{plusone}, $a_{k+1}$ occurs in $w \in W(\{a_2, \dots, a_{n-1}\})$ at most the number of total occurrences in $w$ of higher letters, increased by $1$. As a consequence, each $a_{n-1-i}, 0 \leq i \leq k-2$, occurs in $w$ precisely $2^{i}$ times. We conclude that $w$ has length $L(n-2)$.%, hence the length of $u'$ is $L(n) - L(n-2) - 1 = L(n-2) + 1$.

We may repeat this argument for $a_n$ in order to show that the subword on the %left of $a_n$ has length $L(n-2) +1$ and that on its
right of $a_n$ has length $L(n-2)$. As $L(n) = 2 L(n-2) + 2$, this shows that $a_1$ and $a_n$ sit next to each other, so that $u$ is of the form $w a_1 a_nw'$ for some choice of canonical words $w, w'$ of length $L(n-2)$.
\end{proof}

Notice that this statement is false in the even $n$ case, as the canonical word $a_2 a_3 a_1 a_2 a_4 a_3$ of length $L(4) = 6$ provides an immediate counterexample.

\begin{cor}\label{below2}
Let $n = 2k+1$ be an odd number. Then the number of canonical words of length $L(n)$ in $W(\{a_1, \dots, a_n\})$ equals $\frac{1}{2}\cdot 2^{2^{\lceil{n/2}\rceil}}$.
\end{cor}

\section*{Acknowledgments}
We are grateful to the anonymous referee for suggesting several improvements to the original version of this note.

\vfill

\end{document}